\documentclass[12pt]{amsart}
\usepackage{amsmath,amssymb,amsthm}
\usepackage{enumerate}
\usepackage{graphicx}
\usepackage{amsfonts}
\usepackage{bigints}
\usepackage{amsfonts}
\usepackage{mathptmx}      % use Times fonts if available on your TeX system
\usepackage{graphicx}
\usepackage{amssymb}
\usepackage{amsmath}
\usepackage{bigints}
\usepackage{mathptmx}      % use Times fonts if available on your TeX system
\theoremstyle{plain}%note that this is the default style (to learn)
\newtheorem{theorem}{Theorem}[section]

\newtheorem{corollary}[theorem]{Corollary}

\newtheorem{conjecture}[theorem]{Conjecture}
\theoremstyle{definition}

\newtheorem{remark}[theorem]{Remark}

\makeatletter
\@addtoreset{equation}{section}
\makeatother

\makeatletter
\newcommand{\Spvek}[2][r]{%
	\gdef\@VORNE{1}
	\left(\hskip-\arraycolsep%
	\begin{array}{#1}\vekSp@lten{#2}\end{array}%
	\hskip-\arraycolsep\right)}

\def\vekSp@lten#1{\xvekSp@lten#1;vekL@stLine;}
\def\vekL@stLine{vekL@stLine}
\def\xvekSp@lten#1;{\def\temp{#1}%
	\ifx\temp\vekL@stLine
	\else
	\ifnum\@VORNE=1\gdef\@VORNE{0}
	\else\@arraycr\fi%
	#1%
	\expandafter\xvekSp@lten
	\fi}
\makeatother
\begin{document}
	\title[Orthogonal polynomials associated with equilibrium measures on $\mathbb{R}$]{Orthogonal polynomials associated with equilibrium measures on $\mathbb{R}$}
	\author{G\"{o}kalp Alpan}
\address{Department of Mathematics, Bilkent University, 06800 Ankara, Turkey}
\email{gokalp@fen.bilkent.edu.tr}
\thanks{The author is supported by a grant from T\"{u}bitak: 115F199.}
\subjclass[2010]{31A15 \and 42C05}
\keywords{ Widom factors \and Equilibrium measure \and orthogonal polynomials \and Jacobi matrices}
\begin{abstract}
Let $K$ be a non-polar compact subset of $\mathbb{R}$ and $\mu_K$ denote the equilibrium measure of $K$. Furthermore, let $P_n\left(\cdot; \mu_K\right)$ be the $n$-th monic orthogonal polynomial for $\mu_K$. It is shown that  $\|P_n\left(\cdot; \mu_K\right)\|_{L^2(\mu_K)}$, the Hilbert norm of $P_n\left(\cdot; \mu_K\right)$ in $L^2(\mu_K)$, is bounded below by $\mathrm{Cap}(K)^n$ for each $n\in\mathbb{N}$. A sufficient condition is given for $\displaystyle\left(\|P_n\left(\cdot;\mu_K\right)\|_{L^2(\mu_K)}/\mathrm{Cap}(K)^n\right)_{n=1}^\infty$ to be unbounded. More detailed results are presented for sets which are union of finitely many intervals.
	
\end{abstract}
\maketitle
\section{Introduction and results}
\label{intro}
Let $K$ be an infinite compact subset of $\mathbb{R}$ and let $\|\cdot\|_{L^\infty(K)}$ denote the sup-norm on $K$. The polynomial $T_{n,K}(x)=x^n+\cdots$ satisfying 
\begin{equation}\label{eq1}
\|T_{n,K}\|_{L^\infty(K)}=\min\{\|Q_n\|_{L^\infty(K)}: \mbox{$Q_n$ monic real polynomial of degree $n$}\}
\end{equation} is called the $n$-th Chebyshev polynomial on $K$. We have (see e.g. Corollary 5.5.5 in \cite{Ransford})
\begin{equation}\label{eq2}
\lim_{n\rightarrow\infty}\|T_{n,K}\|_{L^\infty(K)}^{1/n}=\mathrm{Cap}(K),
\end{equation}
where $\mathrm{Cap}(\cdot)$ denotes the logarithmic capacity. Let $M_{n,K}:= \|T_{n,K}\|_{L^\infty(K)}/\mathrm{Cap}(K)^n$. Then $M_{n,K}\geq 2$, see \cite{sif}. If $K=\cup_{i=1}^n [\alpha_i,\beta_i]$ where $-\infty<\alpha_1<\beta_1<\alpha_2<\beta_2\dots<\alpha_n<\beta_n<\infty$, then $(M_{n,K})_{n=1}^\infty$ is bounded and many results were obtained (see \cite{tot11,tot2,tot1,widom2}) regarding the limit points of this sequence. It was recently proved in \cite{csz} that there are Cantor sets for which $(M_{n,K})_{n=1}^\infty$ is bounded. On the other direction, for each sequence $(c_n)_{n=1}^\infty$ of real numbers with subexponential growth, there is a Cantor set $K(\gamma)$ such that $M_{n,K(\gamma)}\geq c_n$ for all $n\in\mathbb{N}$, see \cite{gonchat}. We refer the reader to \cite{sodin} for a general discussion on Chebyshev polynomials and \cite{Ransford,saff} for basic concepts of potential theory.

Throughout the article, by a measure we mean a unit Borel measure with an infinite compact support on $\mathbb{R}$. For such a measure $\mu$, the polynomial $P_n(x;\mu)=x^n+\cdots$ satisfying 
\begin{equation*}\label{eq3}
\|P_n\left(\cdot;\mu\right)\|_{L^2(\mu)}=\min\{\|Q_n\|_{L^2(\mu)}: \mbox{$Q_n$ monic real polynomial of degree $n$}\}
\end{equation*} is called the $n$-th monic orthogonal polynomial for $\mu$ where $\|\cdot\|_{L^2(\mu)}$ is the Hilbert norm in $L^2(\mu)$. Similarly, the polynomial $p_n(x;\mu):=P_n(x;\mu)/\|P_n(\cdot;\mu)\|_{L^2(\mu)}$ is called $n$-th orthonormal polynomial for $\mu$. If we assume that $P_{-1}(x;\mu):=0$ and $P_{0}(x;\mu):=1$ then the monic orthogonal polynomials obey a three term recurrence relation, that is 
\begin{equation}\label{rec}
P_{n+1}(x;\mu) = (x- b_{n+1})P_{n}(x;\mu) - a_n^2 \, P_{n-1}(x;\mu),\,\,\,\,\,\,\,\,n \in \mathbb{N}_0, 
\end{equation}
where $a_n>0$, $b_n\in\mathbb{R}$ and $\mathbb{N}_{0}=\mathbb{N}\cup \{0\}$. We call $(a_n)_{n=1}^\infty$ and $(b_n)_{n=1}^\infty$ as recurrence coefficients for $\mu$. We refer only the $a_n$'s in the text. It is elementary to verify that 
\begin{equation}\label{alar}
\|P_n(\cdot;\mu)\|_{L^2(\mu)}=a_1\cdots a_n 
\end{equation}
for each $n\in\mathbb{N}$. 

For a measure $\mu$, let $W_n(\mu):=\|P_n(\cdot;\mu)\|_{L^2(\mu)}/\mathrm{Cap}(\mathrm{supp(\mu)})^n$ where $\mathrm{supp}(\cdot)$ stands for the support of the measure. By \eqref{eq1}, \eqref{eq2} and using the assumption that $\mu$ is a unit measure, we have 
\begin{equation}\label{eq4}
\|P_n(\cdot;\mu)\|_{L^2(\mu)}\leq \|T_{n,\mathrm{supp}(\mu)}\|_{L^2(\mu)}\leq \|T_{n,\mathrm{supp}(\mu)}\|_{L^\infty(\mathrm{supp}(\mu))}
\end{equation}
for each $n\in\mathbb{N}$. Thus, by \eqref{eq2} it follows that $\lim_{n\rightarrow\infty}\|P_n(\cdot;\mu)\|_{L^2(\mu)}^{1/n}\leq \mathrm{Cap}(\mathrm{supp(\mu)})$. A measure $\mu$ satisfying $\lim_{n\rightarrow\infty}\|P_n(\cdot;\mu)\|_{L^2(\mu)}^{1/n}= \mathrm{Cap}(\mathrm{supp(\mu)})$ is called regular in the sense of Stahl-Totik  and we write $\mu\in\textbf{Reg}$ if $\mu$ is regular. 

For a non-polar compact subset $K$ of $\mathbb{R}$, let $\mu_K$ denote the equilibrium measure of $K$. It is due to Widom that $\mu_K\in\textbf{Reg}$, see \cite{widom} and also \cite{simon1,Stahl,ase2}. Hence,  $\lim_{n\rightarrow\infty} \left(W_n\left(\mu_K\right)\right)^{1/n}=1$ holds. But the behavior of $(W_n(\mu_K))_{n=1}^\infty$ is unknown for many cases and the main aim of this paper is to study the upper and lower bounds of this sequence for general compact sets on $\mathbb{R}$. We remark that by Lemma 1.2.7 in \cite{Stahl} we have $\mathrm{Cap(supp}(\mu_K))= \mathrm{Cap}(K)$, and we use these expressions interchangeably.

A non-polar compact set $K$ on $\mathbb{R}$ which is regular with respect to the Dirichlet problem is called a Parreau-Widom set if $\mathrm{PW}(K):=\sum_j g_K(c_j)$ is finite where $g_K$ denotes the Green function with a pole at infinity for $\overline{\mathbb{C}}\setminus K$ and $\{c_j\}_j$ is the set of critical points of $g_K$. If $K=\cup_{j=1}^n [\alpha_j,\beta_j]$ where $-\infty<\alpha_1<\beta_1<\alpha_2<\beta_2\dots<\alpha_n<\beta_n<\infty$ then $K$ is a Parreau-Widom set and each gap $(\beta_j, \alpha_{j+1})$ contains exactly one critical point $c_j$ and there are no other critical points of $g_K$. Some Cantor sets are Parreau-Widom, see e.g. \cite{alpgon2,peher}. But a Parreau-Widom set is necessarily of positive Lebesgue measure. We refer the reader to \cite{christiansen,yuditskii} for a discussion on Parreau-Widom sets.

Let $K$ be a Parreau-Widom set and $\mu$ be a measure with $\mathrm{supp}(\mu)=K$ which is absolutely continuous with respect to Lebesgue measure, that is $d\mu(t)=\mu^{\prime}(t)\,dt$ on $K$ where $\mu^{\prime}$ is the Radon-Nikodym derivative of $\mu$ with respect to the Lebesgue measure restricted to $K$. Recall that $\mu$ satisfies the Szeg\H{o} condition on $K$ if $\int \log{\mu^{\prime}(t)}\, d\mu_K(t)>-\infty$. In this case we write $\mu\in\mathrm{Sz}(K)$. It is known that $\mu_K\in\mathrm{Sz}(K)$, see Proposition 2 and (4.1) in \cite{christiansen}. By \cite{christiansen}, this implies that there is an $M>0$ such that $1/M<W_n(\mu_K)<M$ holds for all $n\in\mathbb{N}$. In the inverse direction, one can find a Cantor set $K(\gamma)$ such that $W_n\left(\mu_{K\left(\gamma\right)}\right)\rightarrow\infty$ as $n\rightarrow\infty$, see \cite{alpgon}.

First, we restrict our attention to union of several intervals. Let $T_N$ be a real polynomial of degree $N$ with $N\geq 2$ such that it has $N$ real and simple zeros $x_1<\dots <x_n$ and $N-1$ critical points $y_1<\dots<y_{n-1}$ with $|T_N(y_i)|\geq 1$ for each $i\in\{1,\ldots,N-1\}$. We call such a polynomial admissible. If $K=T_N^{-1}([-1,1])$ for an admissible polynomial $T_N$  then $K$ is called a $T$-set. A $T$-set is of the form $\cup_{i=1}^n [\alpha_i,\beta_i]$ with $n\leq N$ where $N$ is the degree of the associated admissible polynomial. For applications of $T$-sets to polynomial inequalities and spectral theory of orthogonal polynomials, we refer the reader to \cite{peherr,totikinv} and Chapter 5 in \cite{Sim3}.  We have the following characterization for $T$-sets, see Lemma 2.2 in \cite{totikacta}:

\begin{theorem}\label{thm1}
	Let $K=\cup_{j=1}^n [\alpha_j,\beta_j]$ be a disjoint union of $n$ intervals. Then $K$ is a $T$-set if and only if $\mu_K([\alpha_j,\beta_j])\in\mathbb{Q}$. If $K=T_N^{-1}{[-1,1]}$ for some admissible polynomial $T_N$ then for each $j\in\{1,\ldots,n\}$ there is an $l\in\mathbb{N}$ such that $\mu_K([\alpha_j,\beta_j])=l/N.$
\end{theorem}

If $K=T_N^{-1}{[-1,1]}$ for an admissible polynomial $T_N$ then (see Theorem 9 and Lemma 3 in \cite{van assche}) since $\mu_K\in\mathrm{Sz}(K)$, there is a sequence $(a_n^\prime)_{n=1}^\infty$ with $a_k^\prime=a_{k+N}^\prime$ for each $k\in\mathbb{N}$ such that $a_n-a_n^\prime\rightarrow 0$ as $n\rightarrow \infty$ where $(a_n)_{n=1}^\infty$ is the sequence of recurrence coefficients in \eqref{rec} for $\mu_K$. In this case we call $(a_n^\prime)_{n=1}^\infty$ the periodic limit for $(a_n)_{n=1}^\infty$ and $(a_n)_{n=1}^\infty$ asymptotically periodic. Our first theorem is about $\left(W_n\left(\mu_K\right)\right)_{n=1}^\infty$ when $K$ is a $T$-set. 

\begin{theorem}\label{thm2}
	Let $K=T_N^{-1}{[-1,1]}$ where $T_N$ is an admissible polynomial with leading coefficient $c$. Furthermore, let $(a_n)_{n=1}^\infty$ be the sequence of recurence coefficients for $\mu_K$ and $(a_n^\prime)_{n=1}^\infty$ be the periodic limit of it. Then
	\begin{enumerate}[(a)]
		\item $\displaystyle\liminf_{n\rightarrow\infty} W_n\left(\mu_K\right)=\sqrt{2}$.
		\item $W_n\left(\mu_K\right)\geq 1$ for each $n\in\mathbb{N}$.
		\item $\displaystyle \inf_l \frac{a_1^\prime\cdots a_l^\prime}{\mathrm{Cap}(K)^l}= \frac{a_1^\prime\cdots a_N^\prime}{\mathrm{Cap}(K)^N} =1.$
	\end{enumerate}
\end{theorem}

An arbitrary compact set $K$ on $\mathbb{R}$ can be approximated in an appropriate way by $T$-sets, see Section 5.8 in \cite{Sim3} and Section 2.4 in \cite{tottot}. We rely upon these techniques in order to prove our main result:

\begin{theorem}\label{thm3}
	Let $K$ be a non-polar compact subset of $\mathbb{R}$. Then $W_n(\mu_K)\geq 1$ for all $n\in\mathbb{N}$.
\end{theorem}

\begin{remark} Theorem \ref{thm3} can be seen as an analogue of Schiefermayr's Theorem (Theorem 2 in\cite{sif}). We do not know if $1$ on the right side of the inequality in  Theorem \ref{thm3} can be improved. This constant can be at most $\sqrt{2}$ by the part $(a)$ of Theorem \ref{thm2}. It suffices to find a bigger lower bound for $W_n\left(\mu_K\right)$ in the part $(b)$ of Theorem \ref{thm2} to improve the result.
\end{remark}

Note that a weaker version of the above theorem was conjectured in \cite{alpgon}. Regularity of $\mu_K$ in the sense of Stahl-Totik follows as a corollary of Theorem \ref{thm3} since the inequality $\lim_{n\rightarrow\infty}\left(W_n\left(\mu_K\right)\right)^{1/n}\geq 1$ directly follows. On the other hand, regularity of a measure $\mu$ in the sense of Stahl-Totik does not even imply that $\limsup_{n\rightarrow\infty} W_n(\mu) >0$, see e.g. Example 1.4 in \cite{simon1}. Hence, the implications of Theorem 3 are profoundly different than those of $\mu_K\in\mathbf{Reg}$. The following result which gives a criterion for unboundedness of $\left(W_n\left(\mu_K\right)\right)_{n=1}^\infty$ is also an immediate corollary of Theorem \ref{thm3}:

\begin{corollary}\label{cor}
	Let $K$ be a non-polar compact subset of $\mathbb{R}$ and $(a_n)_{n=1}^\infty$ be the sequence of recurrence coefficients for $\mu_K$. If $\liminf_{n\rightarrow\infty} a_n =0$ then $\left(W_n\left(\mu_K\right)\right)_{n=1}^\infty$ and $\left(M_{n,K}\right)_{n=1}^\infty$ are unbounded.
\end{corollary}

Corollary \ref{cor} cannot be applied to sets having positive measure since in this case we have $\liminf_{n\rightarrow\infty} a_n >0$, see Remark 4.8 in \cite{alpgon}. There are some sets for which the assumptions in Corollary \ref{cor} hold, see e.g. \cite{alpgon,Barnsley3,Barnsley4}. Apart from these particular examples, there is no criterion on an arbitrary set $K$ on $\mathbb{R}$ (except having positive Lebesgue measure) determining if $\liminf_{n\rightarrow\infty} a_n =0$ for $\mu_K$. It would be interesting to calculate $\liminf_{n\rightarrow\infty} a_n$ for $\mu_{K_0}$ where $K_0$ is the Cantor ternary set. 

To our knowledge, in all known cases when $\left(W_n\left(\mu_K\right)\right)_{n=1}^\infty$ is bounded, $\left(M_{n,K}\right)_{n=1}^\infty$ is also bounded. Thus, it is plausible to make the following conjecture (see also Conjecture 4.2 in \cite{alp3}):
\begin{conjecture}
	Let $K$ be a non-polar compact subset of $\mathbb{R}$. Then $\left(W_n\left(\mu_K\right)\right)_{n=1}^\infty$ is bounded if and only if $\left(M_{n,K}\right)_{n=1}^\infty$ is bounded. 
\end{conjecture}

In Section 2, we present some aspects of Widom's theory and give proofs for the theorems. 
%Convergence of measures should be read as weak-star convergence. 
\section{Proofs}
Let $K=\cup_{j=1}^p[\alpha_j,\beta_j]$ be a disjoint union of several intervals, $E_j:= [\alpha_j,\beta_j]$ for each $j\in\{1,\ldots,p\}$ and $\{c_j\}_{j=1}^{p-1}$ (for $p=1$ there are no critical points) be the set of critical points of $g_K$. Then (see e.g. p. 186 in \cite{peherstor}), we have 
\begin{equation}\label{eqm}
\mu_K^\prime(t)=\frac{1}{\pi}\frac{|q(t)|}{\sqrt{\prod_{j=1}^p\, |(t-\alpha_j)(t-\beta_j)|}},\,\,\,\mbox{    $t\in K$}
\end{equation}
where $q(t)=1$ if $p=1$ and $q(t)=\prod_{j=1}^{p-1}(t-c_j)$ if $p>1$.

Let ${\partial g_K}/{\partial n_+}$ and ${\partial g_K}/{\partial n_-}$ denote the normal derivatives of $g_K$ in the positive and negative direction respectively. These functions are well defined on $K$ except the end points of the intervals. Moreover by symmetry of $K$ with respect to $\mathbb{R}$, we have ${\partial g_K}/{\partial n_+}={\partial g_K}/{\partial n_-}$, see p. 121 in \cite{saff}. Let ${\partial g_K}/{\partial n}:={\partial g_K}/{\partial n_+}$. Then, $({\partial g_K}/{\partial n})(t)= \pi\,\mu_K^\prime(t)$, see (5.6.7) in \cite{Sim3}. This is why we can state the functions and theorems in \cite{widom2} in terms of $\mu_K$ instead of ${\partial g_K}/{\partial n}$. Similarly, instead of harmonic measure at infinity we use the equilibrium measure, since these two measures are the same, see Theorem 4.3.14 in \cite{Ransford}. The concepts that we describe below can be found in \cite{apt,widom2} but with somewhat a different terminology.

Let $\mu\in\mathrm{Sz}(K)$ and $h$ be the harmonic function in $\overline{\mathbb{C}}\setminus K$ having boundary values (nontangential limit exists a.e.) $\log{\mu^{\prime}(t)}$. Then following Section 5 and Section 14 of \cite{widom2}, we define the multivalued analytic function $R$ in $\overline{\mathbb{C}}\setminus K$ by $R(z)=\exp{h(z)+ i \tilde{h}(z)}$ where $\tilde{h}$ is a harmonic conjugate of $h$ and
\begin{equation*}
R(\infty)= \exp{\left(\int \log{\mu^\prime(t)} d\mu_K(t)\right)}.
\end{equation*} 
Now, $R$ has no zeros or poles. Moreover, $|\log{R(z)}|$ is single-valued on $\overline{\mathbb{C}}\setminus K$ and has boundary values  $\log{\mu^\prime(t)}$ on $K$.

Let $F$ be a multivalued meromorphic function having finitely many zeros and poles in $\overline{\mathbb{C}}\setminus K$
for which $|F(z)|$ is single-valued. Then, $$\displaystyle\gamma_j(F):=(1/{2\pi})\underset{E_j}{\triangle}\arg{F},$$  for each $j\in\{1,\ldots,p\}$. Here, $\underset{E_j}{\triangle}\arg{F}$ denotes the increment of the argument of $F$ in going around a positively oriented curve $F_j$ enclosing $E_j$. The curve is taken so close to $E_j$ that it does not intersect with or enclose any points of $E_k$ with $k\neq j$. A multiple-valued function $U$ in $\overline{\mathbb{C}}\setminus K$ with a single-valued absolute value is of class $\Gamma_\gamma$ if $\gamma=(\gamma_1,\ldots,\gamma_p)\in [0,1)^p$ and $\gamma_j(U)= \gamma_j \mod{1}$ for each $j\in\{1,\ldots,p\}$. 

Let $H^2(\overline{\mathbb{C}}\setminus K, \mu^\prime, \Gamma_\gamma)$ denote the space of multi-valued analytic functions $F$ from $\Gamma_\gamma$ in $\overline{\mathbb{C}}\setminus K$ such that $|F(z)^2 R(z)|$ has a harmonic majorant. Then 
\begin{equation*}
\nu(\mu^\prime, \Gamma_\gamma):=\inf_F \int_E |F(t)|^2 \mu^\prime(t) dt.
\end{equation*}
where $F\in H^2(\overline{\mathbb{C}}\setminus K, \mu^\prime, \Gamma_\gamma)$ and $|F(\infty)|=1$.

For the point $(-n\mu_E(E_1) \mod{1}, \ldots, -n\mu_E(E_p)\mod{1})$ with $n\in\mathbb{N}$ we use $\Gamma_n$.

Before giving the proofs, we state some results from \cite{widom2} in a unified way. The part $(a)$ is Theorem 12.3, the part $(c)$ is Theorem 9.2 (see p. 223 for the explanation of why it is applicable) and the part $(b)$ is given in p. 216 in \cite{widom2}.

\begin{theorem}\label{thm4}
	Let $K=\cup_{j=1}^p [\alpha_j,\beta_j]$ be a disjoint union intervals and let $\mu\in\mathrm{Sz}(K)$. Then
	\begin{enumerate}[(a)]
		\item $\left(W_n\left(\mu\right)\right)^2\sim\ \nu(\mu^\prime, \Gamma_n)$ where $a_n\sim\ b_n$ means that $\frac{a_n}{b_n}\rightarrow 1$ as $n\rightarrow\infty$.
		\item $\left(W_n\left(\mu\right)\right)^2\geq \frac{\nu(\mu^\prime, \Gamma_n)}{2}$ for all $n\in\mathbb{N}$.
		\item The limit points of $\left(\left(W_n\left(\mu\right)\right)^2\right)_{n=1}^\infty$ are bounded below by $$2\pi R(\infty)\mathrm{Cap}(K)\exp(-\mathrm{PW}(K)).$$
	\end{enumerate}
\end{theorem}

\begin{proof}[Proof of Theorem \ref{thm2}]
	Let $\{\alpha_j\}_{j}$ and $\{\beta_j\}_j$ be the set of left and right endpoints of the connected components of $K$ respectively so that $\alpha_1<\beta_1<\dots<\alpha_p<\beta_p$. Moreover let $E_j:=[\alpha_j,\beta_j]$ for each $j\in\{1,\ldots,p\}$ and $\{c_j\}_j$ be the set of critical points of $g_K$. 
	\begin{enumerate}[(a)]
		\item First, let us show that $\liminf_{n\rightarrow\infty} \left(W_n\left(\mu_K\right)\right)^2\geq 2$. Since $\mu_K\in\mathrm{Sz}(K)$, Theorem \ref{thm4} is applicable. We need to compute 
		\begin{equation*}
		\log{R(\infty)}= \int \log{\mu_K^\prime(t)} \,d\mu_K(t).
		\end{equation*}
		Using \eqref{eqm}, we can write 
		\begin{equation*}
		\log{R(\infty)}= -\log{\pi}+D_1+D_2+D_3
		\end{equation*}
		where 
		$$D_1=-\frac{1}{2} \sum_{j=1}^p\int \log|t-\alpha_j|\,d\mu_K(t),$$
		$$D_2=-\frac{1}{2} \sum_{j=1}^p\int \log|t-\beta_j|\,d\mu_K(t),$$
		$$D_3=\sum_{j=1}^{p-1}\int \log|t-c_j|\,d\mu_K(t),\,\,\, \mathrm{if} \,\,\,p\geq 2$$
		and $D_3=0$ if $p=1$.
		
		Since $K$ is regular with respect to the Dirichlet problem, $g_K$ can be extended to $\overline{\mathbb{C}}$ by taking $g_K(z)=0$ for $z\in K$ so that $g_K$ is continuous everywhere in $\mathbb{C}$. Besides, 
		\begin{equation}\label{green}
		g_K(z)= -U^{\mu_K}(z) -\log{\mathrm{Cap}(K)}
		\end{equation}
		holds in $\mathbb{C}$ where $U^{\mu_K}(z)=-\int \log{|z-t|}\,d\mu_K(t)$. See p. 53-54 in \cite{saff}.

		By \eqref{green}, for any $z\in K$ we have $\int \log{|z-t|}\,d\mu_K(t)= \log{\mathrm{Cap}(K)}$. Hence, $D_1+D_2= 2p(-1/2)\log{\mathrm{Cap}(K)}=-\log({\mathrm{Cap}(K)^p})$.
		
		For $p\geq 2$, $\int \log|t-c_j|\,d\mu_K(t)=g(c_j)+\log{\mathrm{Cap}(K)}$ by \eqref{green}. Thus, 
		\begin{equation}\label{d3}
		D_3= \mathrm{PW}(K)+\log{\left(\mathrm{Cap}(K)^{p-1}\right)}.
		\end{equation}
		But since $\mathrm{PW}(K)+\log({\mathrm{Cap}(K)^{p-1}})=0$ for $p=1$, \eqref{d3} is valid for $p\geq 1$.
		Therefore,
		\begin{equation*}
		\log{R(\infty)}= -\log{\pi}+\mathrm{PW}(K)-\log{\mathrm{Cap}(K)}.
		\end{equation*}
		Using the part $(c)$ of Theorem \ref{thm4}, we have
		\begin{equation*}
		\liminf_{n\rightarrow\infty} \left(W_n\left(\mu_K\right)\right)^2\geq \frac{2\pi\exp({\mathrm{PW}(K)})\mathrm{Cap}(K)}{\pi\exp({\mathrm{PW}(K)})\mathrm{Cap}(K)}\geq 2.
		\end{equation*}
		In order to complete the proof, it is enough to show that 
		\begin{equation}\label{aaa}
		\liminf_{n\rightarrow\infty} \left(W_n\left(\mu_K\right)\right)^2\leq 2.
		\end{equation}
		
		On $[-1,1]$, we have the formula $p_l(x;\mu_{[-1,1]})= \sqrt{2}S_l(x)$ where $S_l$ is the $l$-th Chebyshev polynomial on $[-1,1]$ of the first kind, see (1.89b) in \cite{rivlin}. By Theorem 1 and Theorem 11 in \cite{van assche} this gives,
		\begin{equation*}
		p_{lN}\left(x;\mu_K\right)= p_l\left(T_N(x);\mu_{[-1,1]}\right)= \sqrt{2} S_l(T_N(x)),
		\end{equation*}
		for each $l\in\mathbb{N}$. The leading coefficient of $p_{lN}\left(x;\mu_K\right)$ is $\sqrt{2}\cdot{2^{l-1}}\cdot c^l$ or in other words $\|P_{lN}(\cdot;\mu_K)\|_{L^2\left(\mu_K\right)}=(\sqrt{2}\cdot{2^{l-1}}\cdot c^l)^{-1}$. By (5.2) in \cite{van assche}, $\mathrm{Cap}(K)^{lN} = (2c)^{-l}$ since (see e.g. p. 135 in \cite{Ransford}) $\mathrm{Cap}[-1,1]=1/2$. Therefore, $W_{lN}(\mu_K)=\sqrt{2}$ for each $l\in\mathbb{N}$ and \eqref{aaa} holds. This completes the proof of the part $(a)$.
		\item By Theorem \ref{thm1}, $(lN+s)\mu_K(E_j) =s\cdot \mu_K(E_j)\mod{1}$ for all $l\in\mathbb{N}$, $s\in\{0,\ldots, N-1\}$ and $j\in\{1,\ldots, N\}$. Hence $\Gamma_{lN+s}= \Gamma_s$ where $l$ and $s$ are as above. Therefore, $\left(\nu\left(\mu_K^\prime,\Gamma_n\right)\right)_{n=1}^\infty$ is a periodic sequence of period $N$. This implies that $\displaystyle\inf_{n\in\mathbb{N}}\nu\left(\mu_K^\prime,\Gamma_n\right)= \liminf_{n\rightarrow\infty}\nu\left(\mu_K^\prime,\Gamma_n\right)$. By the part $(a)$ of Theorem \ref{thm4} and the part $(a)$ of this theorem, we have
		\begin{equation} \label{rrr}
		\liminf_{n\rightarrow\infty}\nu\left(\mu_K^\prime,\Gamma_n\right)=\liminf_{n\rightarrow\infty} \left(W_n\left(\mu_K\right)\right)^2=2.
		\end{equation}
		From \eqref{rrr}, it follows that, $\displaystyle\inf_{n\in\mathbb{N}}\nu\left(\mu_K^\prime,\Gamma_n\right)=2.$ By the part $(b)$ of Theorem \ref{thm4}, we get $\left(W_n\left(\mu_K\right)\right)^2\geq 1$ for each $n\in\mathbb{N}$ which gives the desired result.
		\item Equality on the right can be found in the literature, see e.g. (2.23) in \cite{damanik}. 
		As we see, in the proof of part $(b)$, $\left(W_n\left(\mu_K\right)\right)_{n=1}^\infty$ is asymptotically periodic with the periodic limit $\left(\sqrt{\nu\left(\mu_K^\prime,\Gamma_n\right)}\right)_{n=1}^\infty$. The periodic limit can be written in the form 
		\begin{equation*}
		\left(d\frac{a_1^\prime \cdots a_n^\prime}{\mathrm{Cap}(K)^n}\right)_{n=1}^\infty,
		\end{equation*}
		by Corollary 6.7 of \cite{Chris} where $d\in\mathbb{R}^+$. Since $W_{lN}(\mu_K)=\sqrt{2}$ by the proof of part $(a)$ and $\frac{a_1^\prime \cdots a_{lN}^\prime}{\mathrm{Cap}(K)^{lN}}=1$ holds for all $l\in\mathbb{N}$, we obtain $d=\sqrt{2}$. Besides,
		\begin{equation}\label{rtr}
		\liminf_{l\rightarrow\infty} \sqrt{2}\frac{a_1^\prime \cdots a_l^\prime}{\mathrm{Cap}(K)^l}= \liminf_{l\rightarrow\infty} W_{l}(\mu_K) =\sqrt{2}
		\end{equation}
		holds by the part $(a)$.
		Using periodicity and \eqref{rtr}, we have 
		\begin{equation*}
		\inf_{l\in\mathbb{N}} \frac{a_1^\prime \cdots a_{l}^\prime}{\mathrm{Cap}(K)^{l}}= 	\liminf_{l\rightarrow\infty} \frac{a_1^\prime \cdots a_l^\prime}{\mathrm{Cap}(K)^l}=1.
		\end{equation*}
		This concludes the proof. 
	\end{enumerate}
\end{proof}
\begin{proof}[Proof of Theorem \ref{thm3}]
	By Theorem 5.8.4 in \cite{Sim3}, there is a sequence $(F_s)_{s=1}^\infty$ of $T$-sets such that 
	\begin{equation}\label{ere}
	K\subset \dots\subset F_{s+1}\subset F_{s}\subset \dots \subset \mathbb{R}
	\end{equation}
	and 
	\begin{equation}\label{rer}
	\cap_{s=1}^\infty F_s= K
	\end{equation}
	hold. Moreover, \eqref{ere} and \eqref{rer} imply that
	\begin{equation}\label{weak}
	\mu_{F_s}\rightarrow \mu_K
	\end{equation}
	in weak star sense, and
	\begin{equation*}
	\mathrm{Cap}(F_s)\rightarrow \mathrm{Cap}(K)
	\end{equation*}
	as $s\rightarrow\infty$. 
	
	Let $n\in\mathbb{N}$. Then for each $s\in\mathbb{N},$ we have
	\begin{equation}\label{tete}
	\|P_{n}(\cdot;\mu_{F_s})\|_{L^2\left(\mu_{F_s}\right)}\leq \|P_{n}(\cdot;\mu_{K})\|_{L^2\left(\mu_{F_s}\right)}
	\end{equation}
	by minimality of $P_{n}(x;\mu_{F_s})$ in $L^2\left(\mu_{F_s}\right)$. It follows from monotonicity (see e.g. Theorem 5.1.2 \cite{Ransford}) of capacity that 
	\begin{equation}\label{cap}
	\mathrm{Cap}(K)\leq \mathrm{Cap}(F_s)  \mbox{  for each  } s\in\mathbb{N}.
	\end{equation}
	Hence, 
	\begin{align}
	\left(W_{n}\left(\mu_K\right)\right)^2 &= \frac{\int P_{n}^2(t;\mu_K)\,d\mu_K(t)}{\mathrm{Cap}(K)^{2n}}\\
	\label{aa1}&=\frac{\lim_{s\rightarrow\infty}\int P_{n}^2(t;\mu_K)\,d\mu_{F_s}(t)}{\mathrm{Cap}(K)^{2n}}\\
	\label{aa2}&\geq \liminf_{s\rightarrow\infty} \frac{\int P_{n}^2(t;\mu_{F_s})\,d\mu_{F_s}(t)}{\mathrm{Cap}(F_s)^{2n}}\\
	&=\liminf_{s\rightarrow\infty}\left(W_{n}\left(\mu_{F_s}\right)\right)^2\\
	\label{aa3}&\geq 1.
	\end{align}
	In order to obtain \eqref{aa1}, we use \eqref{weak}. The inequality \eqref{aa2} follows from \eqref{tete} and \eqref{cap}, and \eqref{aa3} is obtained by using the part $(b)$ of Theorem \ref{thm2}. Thus, the proof is complete. 
\end{proof}
\begin{proof}[Proof of Corollary \ref{cor}]
	Let $\left(a_{n_j}\right)_{j=1}^\infty$ be a subsequence of $(a_n)_{n=1}^\infty$ such that $a_{n_j}\rightarrow 0$ as $j\rightarrow\infty$. By \eqref{alar} and Theorem \ref{thm3}, for each $j>1$, we have
	\begin{equation}\label{eee}
	W_{n_j-1}(\mu_K)= W_{n_j}(\mu_K)\frac{\mathrm{Cap}(K)}{a_{n_j}}\geq \frac{\mathrm{Cap}(K)}{a_{n_j}}
	\end{equation}
	Since $a_{n_j}\rightarrow 0$ as $j\rightarrow\infty$, the right hand side of \eqref{eee} goes to infinity as $j\rightarrow\infty$. Hence $\lim_{j\rightarrow\infty}W_{n_j-1}(\mu_K)=\infty$ and in particular $\left(W_n\left(\mu_K\right)\right)_{n=1}^\infty$ is unbounded. Since $\mathrm{supp}(\mu_K)\subset K$,  $\|T_{n,\mathrm{supp}(\mu_K)}\|_{L^\infty(\mathrm{supp}(\mu_K))}\leq \|T_{n,K}\|_{L^\infty(K)}$ holds for all $n\in\mathbb{N}$. Thus, by \eqref{eq4}, $W_n\left(\mu_K\right)\leq M_{n,K}$ for each $n\in\mathbb{N}$. This implies that $\left(M_{n,K}\right)_{n=1}^\infty$ is also unbounded. 
\end{proof}

%\begin{acknowledgements}
%\end{acknowledgements}


\begin{thebibliography}{155}
	
	\bibitem{alpgon}Alpan, G., Goncharov, A.: {Orthogonal polynomials for the weakly equilibrium Cantor sets}, accepted for publication in Proc. Amer. Math. Soc. 
	
	\bibitem{alpgon2}Alpan, G., Goncharov, A.: {Orthogonal polynomials on generalized Julia sets}, Preprint (2015), arXiv:1503.07098v3
	
	\bibitem{alp3}Alpan, G., Goncharov, A., \c{S}\.{i}m\c{s}ek, A.N.: {Asymptotic properties of Jacobi matrices for a family of fractal measures}, Preprint (2016), arXiv:1603.02312v1
	
	\bibitem{apt}Aptekarev, A.I.: {Asymptotic properties of polynomials orthogonal on a system of contours, and periodic motions of Toda lattices}, Mat. Sb., \textbf{125},  231–258 (1984), English translations in Math. USSR Sb., \textbf{53} , 233–-260 (1986)
	
	\bibitem{Barnsley3}Barnsley M.F., Geronimo, J.S., Harrington, A.N.: {Infinite-Dimensional Jacobi Matrices Associated with Julia Sets}, Proc. Amer. Math. Soc., \textbf{88}(4), 625--630 (1983)
	
	\bibitem{Barnsley4}Barnsley, M.F., Geronimo, J.S., Harrington, A.N.: {Almost periodic Jacobi matrices associated with Julia sets for polynomials}, Comm. Math. Phys., \textbf{99}(3), 303--317 (1985)
	
	\bibitem{christiansen}Christiansen, J.S.: {Szeg\H{o}'s theorem on Parreau-Widom sets}, Adv. Math., \textbf{229}, 1180--1204 (2012)
	
	
	
	\bibitem{Chris}Christiansen, J.S., Simon, B., Zinchenko, M.: {Finite Gap Jacobi Matrices, II. The Szeg\"o Class}, Constr. Approx., \textbf{33}, 365--403 (2011)
	
	\bibitem{csz}Christiansen, J.S., Simon, B., Zinchenko, M.: {Asymptotics of Chebyshev Polynomials, I. Subsets of $\mathbb{R}$}, Preprint (2015), arXiv:1505.02604v1
	
	\bibitem{damanik}Damanik, D, Killip, R., Simon, B.: {Perturbations of orthogonal polynomials with periodic recursion coefficients}, Ann. Math., \textbf{171}, 1931--2010 (2010)
	
	
	\bibitem{van assche}Geronimo, J.S., Van Assche, W.: {Orthogonal polynomials on several intervals via a polynomial mapping}, Trans. Amer. Math. Soc., \textbf{308}, 559--581 (1988)
	
	\bibitem{gonchat} Goncharov A., Hatino\u{g}lu, B.: {Widom Factors}, Potential Anal., \textbf{42}, 671--680 (2015)
	
	\bibitem{peherr}Peherstorfer, F.: {Orthogonal and extremal polynomials on several intervals}, J. Comput. Appl. Math., \textbf{48}, 187--205 (1993) 
	
	\bibitem{peherstor}Peherstorfer, F.: {Deformation of Minimal Polynomials and Approximation of Several Intervals by an Inverse Polynomial Mapping}, J. Approx. Theory, \textbf{111}, 180--195 (2001)
	
	\bibitem{peher} Peherstorfer, F., Yuditskii, P.: {Asymptotic behavior of polynomials orthonormal on a homogeneous set}, J. Anal. Math., \textbf{89}, 113-154 (2003)
	
	\bibitem{Ransford}Ransford, T.: {Potential theory in the complex plane}, Cambridge University Press, (1995)
	
	\bibitem{rivlin} Rivlin, T.J.: {Chebyshev polynomials : from approximation theory to algebra and number theory}, Second Edition, J. Wiley and Sons, New York, (1990)
	
	\bibitem{saff}Saff, E.B., Totik, V.: {Logarithmic potentials with external fields}, Springer-Verlag, New York (1997)
	
	\bibitem{sif} Schiefermayr, K.: {A lower bound for the minimum deviation of the Chebyshev polynomial on a compact real set}, East J. Approx., \textbf{14}, 223--233 (2008)
	
	\bibitem{simon1}Simon, B.: {Equilibrium measures and capacities in spectral theory}, Inverse Probl. Imaging, \textbf{1}, 713--772 (2007)
	
	\bibitem{Sim3}Simon, B.: {Szeg\H{o}'s Theorem and Its Descendants: Spectral Theory for $L^2$ Perturbations of Orthogonal Polynomials}, Princeton University Press, Princeton, NY (2011)
	
	\bibitem{sodin}Sodin, M., Yuditskii, P.: {Functions deviating least from zero on closed subsets of the real axis}. St. Petersbg. Math. J. \textbf{4}, 201--249 (1993)
	
	\bibitem{Stahl}Stahl, H., Totik, V.: {General orthogonal polynomials},
	Encyclopedia of Mathematics, vol. 43, Cambridge University Press, New York (1992)
	
	\bibitem{tottot}Totik, V.:{Asymptotics for Christoffel functions for general measures on the real line}, J. Anal. Math., \textbf{81}, 283--303 (2000)
	
	\bibitem{totikacta}Totik, V.: {Polynomials inverse images and polynomial inequalities}, Acta Math., \textbf{187}, 139--160 (2001)
	
	\bibitem{tot11} Totik, V.: {Chebyshev constants and the inheritance problem}, J. Approx. Theory, \textbf{160}, 187--201 (2009)
	
	\bibitem{totikinv} Totik, V.: {The polynomial inverse image method}. In \textit{Springer Proceedings in Mathematics},  Approximation Theory XIII: San Antonio, \textbf{13}, M. Neamtu and L. Schumaker (eds.), 345--367 (2010)
	
	\bibitem{tot2} Totik, V.: {Chebyshev Polynomials on Compact Sets}, Potential Anal., \textbf{40}, 511--524 (2014)
	
	\bibitem{tot1} Totik, V., Yuditskii, P.: {On a conjecture of Widom}, J. Approx. Theory, \textbf{190}, 50--61 (2015)
	
	
	
	\bibitem{ase2}Van Assche, W.: {Invariant zero behaviour for orthgonal polynomials on compact sets of the real line}, Bull. Soc. Math. Belg. Ser. B \textbf{38}, 1--13 (1986)
	
	
	\bibitem{widom}Widom, H.: {Polynomials associated with measures in the complex plane}, J. Math. Mech, \textbf{16}, 997--1013 (1967)
	
	\bibitem{widom2}Widom, H: {Extremal polynomials associated with a system of curves in the complex plane}, Adv. Math., \textbf{3}, 127--232 (1969)
	
	\bibitem{yuditskii}Yudistkii, P: {On the Direct Cauchy Theorem in Widom Domains: Positive and Negative Examples}, Comput. Methods Funct. Theory, \textbf{11}, 395--414 (2012)
	% etc
\end{thebibliography}


\end{document}